\documentclass[11pt]{amsart}

\usepackage{epigamath}

\usepackage[english]{babel}

\numberwithin{equation}{section}

\usepackage{macro-longdash}

\usepackage[shortlabels]{enumitem}
\setlist[enumerate,1]{label={\rm(\arabic*)}, ref={\rm\arabic*}} 

\usepackage{mathtools}
\usepackage[all]{xy}
\usepackage[utf8]{inputenc}
\usepackage{varioref}
\usepackage[normalem]{ulem}
\usepackage{amsfonts}
\usepackage{esint}
\usepackage{tikz}
\usepackage{empheq}
\usepackage{tikz-cd}
\usepackage{subcaption}

\usetikzlibrary{matrix,arrows,decorations.pathmorphing}
\usepackage{mathrsfs} % not sure 
\usepackage{cleveref}

\usepackage{soul} 

\usepackage{framed}

\theoremstyle{plain}
\newtheorem{theorem}{Theorem}[section]

\newtheorem{lemma}[theorem]{Lemma}
\newtheorem{corollary}[theorem]{Corollary}
\newtheorem{proposition}[theorem]{Proposition}

\theoremstyle{definition}
\newtheorem{definition}[theorem]{Definition}

\theoremstyle{remark}

\newtheorem{remark}[theorem]{Remark}

\newcommand{\scrG}{{\mathscr G}}
\def\Mon{\operatorname{Mon}}

\def\Branch{\operatorname{Branch}}
\def\II{\operatorname{II}}
\renewcommand{\P}{{\mathbb P}}

\newcommand{\sO}{\mathcal O}

\newcommand{\wh}[1]{{\widehat{#1}}}
\newcommand{\wt}[1]{{\widetilde{#1}}}

\newcommand{\Bl}{\operatorname{Bl}}
\newcommand{\Gr}{\operatorname{Gr}}
\newcommand{\Hilb}{\operatorname{Hilb}}
\newcommand{\Sym}{\operatorname{Sym}}
\newcommand{\Gal}{\operatorname{Gal}}
\newcommand{\closure}{\operatorname{closure}}

\newcommand{\x}{\times}

\newcommand{\et}{{\mathrm{et}}}
\newcommand{\flatt}{{\mathrm{flat}}}
\newcommand{\pr}{{\mathrm{pr}}}
\newcommand{\gtc}{{\mathrm{gtc}}}
\newcommand{\red}{{\mathrm{red}}}
\newcommand{\aCM}{{\mathrm{aCM}}}
\newcommand{\Stein}{{\mathrm{Stein}}}
\newcommand{\sing}{{\mathrm{sing}}}

\newcommand{\addappendix}{%
  \phantomsection % for correct page number in toc
  \appendix
  \setcounter{section}{0}
  \section*{\appendixname}% start the appendix
  \addcontentsline{toc}{section}{\appendixname}% add it to the toc
  \setcounter{section}{1}
  \setcounter{theorem}{0}
}

\newcommand{\supth}[1]{\ensuremath{#1^{\mathrm{th}}}}

\EpigaVolumeYear{10}{2026} \EpigaArticleNr{14} \ReceivedOn{February 4, 2025}
\InFinalFormOn{Novembre 19, 2025}
\AcceptedOn{February 26, 2026}

\title{Lines, twisted cubics on cubic fourfolds, and the monodromy of the Voisin map}
\titlemark{Monodromy of the Voisin map}

\author{Franco Giovenzana}
\address{Laboratoire de Math\'ematiques d'Orsay, Universit\'e Paris-Saclay, Rue Michel Magat, B\^at. 307, 91405, Orsay, France}
\email{franco.giovenzana@universite-paris-saclay.fr}
\author{Luca Giovenzana}
\address{Department of Mathematics, Universiteit Antwerpen, Middelheimlaan 1, 2020 Antwerpen, Belgium}
\email{luca.giovenzana@uantwerpen.be}

\authormark{F.~Giovenzana and L.~Giovenzana}

\AbstractInEnglish{For a general cubic fourfold \( Y \) with associated Fano variety of lines \( F \), we show that the monodromy group of the finite degree 16 rational Voisin self-map \( \psi \colon F \dashrightarrow F \) is maximal.
To achieve this, we investigate the intriguing interplay between \( \psi \) and the fixed locus of the antisymplectic involution on the LLSvS variety \( Z \), examined via the degree 6 Voisin map \( F \times F \dashrightarrow Z \).}


\acknowledgement{Franco Giovenzana was funded by Deutsche Forschungsgemeinschaft (DFG, German Research Foundation) Projektnummer 509501007, and partially supported by the European Research Council (ERC) under the European Union's Horizon 2020 research and innovation programme (ERC-2020- SyG-854361- HyperK). Both authors are members of INdAM GNSAGA.}

\begin{document}

%%%%%%%%%%%%%%%%%%%%%%%%%%%%%%%
% Title page
%%%%%%%%%%%%%%%%%%%%%%%%%%%%%%%

%\removeabove{}
%\removebetween{}
%\removebelow{}

\maketitle

\begin{prelims}

\DisplayAbstractInEnglish

\bigskip

\DisplayKeyWords

\medskip

\DisplayMSCclass

\end{prelims}

%%%%%%%%%%%%%%%%%%%%%
% Table of Contents
%%%%%%%%%%%%%%%%%%%%%

\newpage

\setcounter{tocdepth}{1}

\tableofcontents

%%%%%%%%%%%%%%%%%%%%%
% Content begins here
%%%%%%%%%%%%%%%%%%%%%

\section{Introduction}
Let $Y$ be a smooth cubic fourfold and $F$ be its (Fano) variety of lines. The rich geometry of these varieties has attracted great attention from the mathematical community for several reasons. Notably, the Fano variety~$F$ is one of the earliest examples of a locally complete family of projective hyperk\"ahler manifolds. It is deformation equivalent to the Hilbert square of a K3 surface, and it has a natural polarisation of degree~6 and divisibility 2; see \cite{BD}. The duo $Y-F$ between a Fano variety and a hyperk\"ahler manifold has inspired numerous other constructions.

One distinguishing feature of $F$ is the degree 16 rational Voisin self-map 
$\psi\colon F\dashrightarrow F$. Remarkably, 
$\psi$ is the only known finite, non-birational self-map defined on a locally complete family of projective hyperk\"ahler manifolds. Finite maps such as $\psi$ are particularly subtle to study as, unlike birational maps, they cannot be detected through their action on the second cohomology group and lack a comprehensive classification result like the Hodge Torelli theorem that would allow precise control over their behaviour. Since its construction, see \cite{voisin-map-F}, this map has continued to attract significant interest in the mathematical community, as demonstrated by the numerous works on the subject \cite{Amerik,GK-invariants, GK-monodromy, GK-lines}.

In this paper we focus on the monodromy of the Voisin map $\psi$. The monodromy group is a discrete invariant that encodes the symmetries and intrinsic geometry of finite maps. Determining the monodromy group of branched coverings is a classical problem, originating with Jordan in the 1870s. This topic was revitalised by Harris, who provided a modern framework by proving that the monodromy and Galois groups coincide, and who developed tools to establish when the monodromy group is maximal. Significant progress was made by Vakil, who introduced innovative techniques to study monodromy in Schubert problems on Grassmannians; see \cite{vakil}.

Building on these developments, we take a step further by investigating linear spaces on a cubic hypersurface of dimension 4. Our main result shows that given a general cubic fourfold, the monodromy group of the Voisin map $\psi$ is maximal, meaning that it is the full symmetric group on 16 elements. 
To achieve this, we leverage another Voisin map involving twisted cubics on the cubic fourfold, which unveils a fascinating connection with the fixed locus of the natural anti-symplectic involution acting on the LLSvS variety. 
\bigskip

We now introduce the necessary notation and recall some useful results in order to state our main theorems.
Lines on a cubic fourfold fall into two cases: for a general line $L$, the linear space
\[
\Lambda_L := \bigcap_{p\in L} T_p Y
\]
is 2-dimensional, and in this case the line is said to be of \textit{type I}. For special lines, called of \textit{type II}, the dimension of $\Lambda_L$ is 3; see Proposition~\ref{prop:Lines} and Definition~\ref{def:Types} for more details.
For the general line we have $\Lambda_L \cap Y = 2L + R$ for a line $R\in F$, and one sets $\psi(L) := R$. This map has been studied in detail in \cite{Amerik}, its indeterminacy locus consists of the lines of type $\II$, which for a general cubic fourfold form a smooth surface $S_{\II}$. Blowing up $S_{\II}$ provides a resolution of the indeterminacy
\[
\xymatrix{
&\wh F\ar[ld]\ar[rd]^{\wh \psi}\\
F\ar@{-->}[rr]^{\psi} && F\rlap{.}
}
\]
Gounelas and Kouvidakis recently computed that the restriction of $\wh \psi$ to the exceptional divisor $E$ of the blow-up morphism, which coincides with the ramification divisor of $\wh\psi$,  is birational onto the image.

In Section~\ref{sec:ramification} we study the map $\wh\psi$ by considering the projection of $Y$ from a general line $R$ in $Y$. Its resolution is a conic bundle over $\mathbb P^3$ with discriminant locus a quintic surface $S_R$ with exactly 16 nodes. These nodes correspond to the preimage of $R$ under $\wh\psi$.
Nodal quintic surfaces have been classically studied by Beauville and Catanese \cite{Beauville-nodal, Catanese} and more recently studied in \cite{HUY-nodal-quintics,7auth, catanese-new}.
First we study the singularities of $S_R$ for a special line $R$.

\begin{theorem}[see Theorem~\ref{thm:A3-singularity-quintic}]
Let $R\in F$ be a general line in the branch divisor of\, $\wh\psi$. Then the quintic surface $S_R$ has one singularity of type $A_3$, corresponding to the unique line $L$ of type $\II$ over $R$.
\end{theorem}

In Section~\ref{sec:variety-P} we shift our focus to the geometry of twisted cubics. With any smooth cubic fourfold $Y$ not containing a plane, Lehn, Lehn, Sorger, and van Straten associated an $8$-dimensional hyperk\"ahler variety~$Z$ parametrising families of twisted cubics and their flat degenerations. The variety $Z$ is equipped with a natural antisymplectic involution $\tau$; see \cite{Lehn-oberwolfach}.
Its fixed locus is a smooth Lagrangian submanifold with two connected components: one is isomorphic to the cubic fourfold $Y$; the other one, $W$, is of general type, see \cite{FMSOG-II}, and remains somewhat mysterious. Using the degree 6 rational map $\varphi\colon F\times F \dashrightarrow Z$ constructed by Voisin \cite{Voisin-map-varphi}, we offer an alternative description of $W$.
We define the variety $P$ as the closure in $F\times F$ of 
\[
\{ (L_1,L_2)\in F\times F \mid L_i \text{ are of type I, $L_1\not = L_2$, and }\psi(L_1)=\psi(L_2)\},
  \]
which is birational to an irreducible component of the self-product of $\wh F$ over $F$.

\begin{theorem}[see~Theorem~\ref{thm: P->W}]
The variety $P$ is mapped onto $W$ under the Voisin map $\varphi\colon F\times F \dashrightarrow Z$.
\end{theorem}

In Section~\ref{sec:monodromy}, after revising the basic notions of monodromy, we tackle the study of $\wh \psi$. Despite extensive study and numerous works concerning $\wh\psi$, for example about its entropy, see \cite{Amerik}, many of its properties remain elusive. Through an investigation of the restriction of the  map $\varphi$ to the variety $P$, we prove the following.

\begin{theorem}[see Theorem~\ref{thm:monodromy-maximal}]\label{thm:Main}
The monodromy group of $\psi$ is the entire symmetric group on $16$ elements.
\end{theorem}

%%%%%%%%%%%%%%%%%%%%%%%%%%
%%%%% ACKNOWLDGEMENTS
%%%%%%%%%%%%%%%%%%%%%%%%%%

\subsection*{Acknowledgments} This project began a long time ago, and over the years we benefited from conversations with many people. It is our pleasure to thank everybody who expressed interest and shared their point of view, especially Enrico Fatighenti, Frank Gounelas, Christian Lehn, Emanuele Macr\`\i, Giovanni Mongardi, Alan Thompson, and Yilong Zhang.
We also express our gratitude to the anonymous referees for their careful reading.

\clearpage
\phantomsection
\section{Nodal quintic surfaces with 16 nodes}\label{sec:ramification}
In this section we recall basic facts about lines on cubic fourfolds and various properties of the Voisin map on $F$. Then we move on to studying degenerations of nodal quintic surfaces with an even set of 16 nodes naturally arising in this context.

Recall that the Gauss map associates to any point of the smooth cubic fourfold $Y\subset \P(V)$ its projective tangent space:  
\[
\mathscr G\colon Y \longrightarrow \P\left(V^\vee\right),\quad P\longmapsto T_P Y \simeq \P^4.
\]
Clemens and Griffiths  distinguished lines on cubic hypersurfaces into two types. We recall here the definition for cubic fourfolds; see \cite{griffiths-clemens}.

\begin{proposition}\label{prop:Lines}
Given a line $L$ on a smooth cubic fourfold $Y$, either the following equivalent conditions hold:
\begin{enumerate}
    \item\label{p:L-1} $N_{L|Y} \simeq \sO_L^{\oplus 2}\oplus \sO_L(1)$,
    \item $\mathscr G|_L\colon L\to \mathscr G (L)$ is $1:1$,
    \item $\mathscr G (L)$ is a smooth conic in $\P(V^\vee)$,
    \item\label{p:L-4} $\bigcap_{P\in L} T_P Y$ is isomorphic to $\P^2$;
\end{enumerate}
or the following equivalent conditions hold:
\begin{enumerate}[resume]
    \item\label{p:L-5} $N_{L|Y} \simeq \sO_L(-1)\oplus \sO_L(1)^{\oplus 2}$,
    \item $\mathscr G|_L\colon L\to \mathscr G (L)$ is $2:1$,
    \item $\mathscr G (L)$ is a line in $\P(V^\vee)$,
    \item\label{p:L-8} $\bigcap_{P\in L} T_P Y$ is isomorphic to $\P^3$.
\end{enumerate}
\end{proposition}

\begin{definition}[\textit{cf.} {\cite[Definition~6.6]{griffiths-clemens}}]\label{def:Types}
Given a line $L$, we say that $L$ is a line of type I if the equivalent conditions~\eqref{p:L-1}--\eqref{p:L-4} hold, whereas if~\eqref{p:L-5}--\eqref{p:L-8} hold, we say that $L$ is a line of type II.  We set $\Lambda_L:= \cap_{P\in L} T_P Y$.
\end{definition}

We record here  elementary observations for future reference.

\begin{remark}
If $L\in F$ is of type $\II$, then the line $\mathscr G(L) \subset \P(V^\vee)$ is the projective dual of $ \Lambda_L \subset \P(V)$. In the case of a line of type I, $\mathscr G(L)$ spans a $\P^2$ in $\P(V^\vee)$ which is dual to $\Lambda_L$.
\end{remark}

\begin{remark}[\textit{cf.} {\cite[Remark~2.2.2]{HuyBookCubics}}]\label{huy-derivatives}
  Let $Y = V(f)\subset \P^5$ be a smooth cubic fourfold, and let $L$ be a line in $Y$. Then $L$ is of type II if and only if the partial derivatives $\partial_0 f|_L,\ldots,\partial_5 f|_L$ span a vector space of dimension 2 in $H^0(L,\sO_{\mathbb P^5}(2))$.
\end{remark}

Voisin observed that if $L$ is a line of type I, then $\Lambda_L\cap Y$ is a cubic curve, which, as it contains $L$ with multiplicity 2, consists of $L$ and a residual line $R$. As the general line is of type I, one gets the following. 

\begin{definition}
    The Voisin map $\psi$ is the rational map defined by mapping a line $L$ of type I to its residual $R$
\begin{align*}
\psi\colon F \longdashrightarrow F, \quad
L \longmapsto R.
\end{align*}
\end{definition}

In general, we say that the line $R$ is \emph{residual} to $L$ if there exists a projective plane $\Lambda$ such that $\Lambda\cap Y= 2L + R$. Following Gounelas and Kouvidakis, if a line is residual to itself, we refer to it as a \emph{triple} line. For the general cubic fourfold, the locus of triple lines is an irreducible surface of general type; see \cite[Theorem~A]{GK-lines}. We denote it by $V$.

The map $\psi$ has been studied in \cite{voisin-map-F} and \cite{Amerik}, where it is proven to be finite of degree 16. For a general cubic fourfold, lines of type II form a smooth surface $S_{\II}$ in $F$, and blowing up $F$ in this surface resolves the indeterminacy of $\psi$:
\begin{align}\label{diag:Res}
\xymatrix{
&\wh F\ar[ld]_{\pr}\ar[rd]^{\wh \psi}\\
F\ar@{-->}[rr]^{\psi} && F\rlap{.}
}
\end{align}

The blow-up $\wh F$ can be identified as the closure of the following rational map (see \cite[Lemma 4.1]{GK-lines} and \cite[Remark 2.2.19]{HuyBookCubics}):
\begin{align*}
F \longdashrightarrow \Gr(3,6), \quad L \longmapsto \Lambda_L.
\end{align*}
Elements in the exceptional locus are then just pairs $(L,\Xi)$, where $L$ is a line of type II and $\Xi$ is a projective plane such that
$L \subset \Xi \subset \Lambda_L$. For such 2-dimensional spaces $\Xi$,  we have that $\Xi \cap Y = 2L + R$; in particular, $R$ is residual to $L$. 

Given a general line $R\in F$, we consider the diagram
\[\xymatrix{
&\wt Y:=\Bl_R Y\ar[ld]_p\ar[rd]^{\wt \pi}\\
Y \ar@{-->}[rr]^{\pi_R} && \P^3\rlap{,}
}
\]
where $\pi_R$ is the projection from $R$ and $\wt Y$ is the blow-up of $Y$ in $R$. The morphism $\wt \pi$ is a conic bundle  with discriminant a quintic surface $S_R$, whose singular locus consists of 16 nodes. The 16 nodes correspond to the preimage of $R$ under $\wh\psi$. Indeed, if $p_i$ denote the nodes of $S_R$ for $i=1,\ldots,16$, then the $L_i:=p(\wt \pi^{-1}(p_i))$ are the 16 lines for which $\psi(L_i) = R$ (see \cite[Section~6.4.5]{HuyBookCubics} for an account of the various results in the literature about this surface).

Let $E$ be the exceptional divisor of the blow-up morphism $\wh F \to F$. As $F$ has trivial canonical bundle, the divisor $E$ coincides with the ramification locus of the map $\wh \psi$.

\begin{theorem}[\textit{cf.} {\cite[Theorem B]{GK-monodromy}} ]\label{GK-RamBirational}
Let $Y$ be a general cubic fourfold.  The restriction $\wh \psi|_E \colon E \to F$ is generically $1\!$-to-$1$ onto the image.
\end{theorem}

\begin{proposition}\label{cor:ram-simple}
    Let $Y$ be a general cubic fourfold. The ramification at the general point of the ramification locus of\, $\wh\psi$ is simple.
\end{proposition}

\begin{proof}
    Comparing the pullbacks of the canonical of $F$ in Diagram~\eqref{diag:Res}, one gets
    \begin{align*}
        E = \pr^*(K_F) + E = K_{\widehat F}= \widehat\psi^*(K_F) + R = R,
    \end{align*}
    where $R$ is the ramification locus with the natural schematic structure. Hence the ramification is simple by \cite[Lemma~I.16.1]{BHPvV}
\end{proof}

In other words, over the general point $R$ in the branch divisor of $\wh \psi$, there is exactly one point of ramification; \textit{i.e.}~there exists exactly one line $L$ of type II with residual $R$, meaning that $\wh\psi (L, \langle L,R \rangle) = R$.

\begin{theorem}\label{thm:A3-singularity-quintic}
    Let $R\in F$ be a general line in the branch divisor of\, $\wh \psi$. Then the quintic surface $S_R$ has exactly $14$ singularities of type $A_1$ and exactly $1$ singularity of type $A_3$, corresponding to the unique line $L$ of type $\II$ with residual line $R$.
\end{theorem}

\begin{proof}
    Let $R$ be a general line in the branch locus of $\wh\psi$, and let $L$ be the unique line of type $\II$ with residual~$R$. We may assume that $R$ is of type I and  that it is given as the zero set $V(x_2,x_3,x_4,x_5)$. Following Clemens and Griffiths \cite[Equation~6.9]{griffiths-clemens}, we write the equation of the cubic fourfold as
    \begin{equation}
       f:= x_4x_0^2 + x_5x_0x_1 + x_3x_1^2 + x_0Q_0 + x_1Q_1 + P = 0,
    \end{equation}
    where $Q_0$, $Q_1$, and $P$ are homogeneous polynomials in the variables $x_2$, $x_3$ ,$x_4$, $x_5$ of degree 2, 2, and 3, respectively. 
    The quintic surface $S_R\subset V(x_0,x_1) \simeq \mathbb P^3$ is given by the determinant of the matrix
    \[
    M:=\begin{pmatrix}
P & Q_0 & Q_1\\
Q_0 & x_4 & x_5 \\
Q_1 & x_5 & x_3
\end{pmatrix}.
\]

We assume that the line $L$ is given by $V(x_1,x_2,x_4,x_5)$;
then $p := (1:0:0:0:0:0)$ is the intersection of~$L$ and $R$, and $q:=(0:0:0:1:0:0)$ is the point of intersection with $\mathbb P^3$.

We translate these assumptions in conditions on the coefficients of $Q_0$, $Q_1$, and $P$.
For this, we introduce the notation $Q_0 = \sum a_I\underline{x}^I$, $Q_1 = \sum b_I\underline x^I$, and $P = \sum c_I\underline x^I$.

The cubic fourfold $Y$ contains the line $L$ if $f(x_0,0,0,x_3,0,0,0)$ is the zero form in $\mathbb C[x_0,x_3]$; hence we get the condition
\begin{align}\label{eq: cond1}
a_{33} = c_{333} = 0.
\end{align}

The fact that the line $R$ is residual to $L$, \textit{i.e.}~$\langle L, R\rangle \cap Y = 2L+R$, translates into the further condition
\begin{align}\label{eq: cond2}
b_{33} = 0.
\end{align}

By Remark~\ref{huy-derivatives} the fact that the line $L$ is of type II can be rephrased in terms of the derivatives $\partial_i f|_L$, and we get the additional condition
\begin{align}\label{eq: cond3}
 a_{23}c_{335} - a_{35}c_{233} = 0.
\end{align}

To conclude the proof, we check with the aid of Macaulay2 that the surface defined by $\det(M)$ for polynomials $Q_0$, $Q_1$, and $P$ with coefficients satisfying conditions~\eqref{eq: cond1},~\eqref{eq: cond2}, and \eqref{eq: cond3}
has a singularity of type $A_3$ at the point $q$. See the ancillary file for the code.  
\end{proof}

As a consequence, we obtain an alternative proof of the simpleness of the ramification of $\widehat\psi$.

\begin{proof}[Proof of Proposition~\ref{cor:ram-simple}]
    Let $b$ be a general point in the branch divisor of $\psi$. Then by Theorem~\ref{GK-RamBirational} the preimage $\wh\psi^{-1}(b)$ consists of a finite number of points $p_1,\ldots, p_k$, where only $p_1$ is of ramification, while $\wh\psi$ is \'etale at the points $p_2,\ldots, p_k$. All the points $p_i$ correspond to singularities of the quintic surface~$S_b$: in particular, $p_2,\ldots, p_k$ are singularities of type $A_1$; in contrast, $p_1$ is a singular point of type $A_3$ by Theorem~\ref{thm:A3-singularity-quintic}. 
    The surface $S_b$ is a specialization of the surface $S_R$ for the general line $R\in F$, which has exactly 16 nodes. Thus we deduce that 2 points among these 16 specialize to the singularity of type $A_3$, whereas the remaining 14 specialize to $p_2,\ldots,p_k$, deducing that $S_b$ has 14 singularities of type $A_1$ apart from the $A_3$ singularity.
\end{proof}

Let $R$ be any line on the cubic fourfold. We denote by $F_R$ the locus of lines meeting $R$: 
\[
F_R :=\lbrace L \in F : L\cap R \not = \emptyset \rbrace.
\]
If $R$ is general, then $F_R$ is a smooth surface; see \cite[Lemma~3.1]{voisin-torelli}. As an application of our analysis, we study the singularities of $F_R$ for a general point $R$ in the branch locus of $\wh\psi$.

\begin{corollary}\label{cor:SurfLinesBranch}
Let $R\in F$ be a general line in the branch locus of\, $\wh \psi$. If the surface $F_R$ is normal, then it has exactly one singular point, which is of type $A_1$ and corresponds to the unique line $L$ of type $\II$ with residual line $R$.
\end{corollary}

\begin{proof}
    Let $\pi_R\colon Y \dashrightarrow \P^3$ be the projection from $R$. Then there is a map 
    \[f\colon F_R \longdashrightarrow S_R,\quad L \longmapsto \pi_R(L).\]
    The map clearly extends to the point $R$ by mapping $R$ to $\pi_R(\Lambda_R)$, when the line $R$ is of type $I$. The map~$f$ is 2-to-1 and can be regarded as the quotient of an involution, which, in case $R$ is general, has 16 fixed points mapped to the 16 singular points of $S_R$.
    If $R$ is as in the hypothesis, then the fixed points of this involution are in bijection with the singularities of $S_R$: 14 singularities of type $A_1$ and 1 singularity of type $A_3$ corresponding to
    the unique line $L$  of type $\II$ with residual $R$ by Theorem~\ref{thm:A3-singularity-quintic}.
    
    As $F_R$ is normal, the morphism $f$ around each singular point is the unique morphism characterized by \cite[Proposition~3.13]{GKP}. Locally around the $A_1$ singular point, it is isomorphic to the quotient presentation  $\mathbb C^2\to \mathbb C^2/\mu_2$ of the $A_1$ singularity, and around the $A_3$ singular point, it is isomorphic to the presentation $\mathbb C^2/\mu_2\to\mathbb C^2/\mu_4$ of the $A_3$ singularity as quotient of the $A_1$ singularity.
\end{proof}

%%%%%%%%%%%%%%%%%%%%%%%%%%%%%%%%
% Second Lagrangian
%%%%%%%%%%%%%%%%%%%%%%%%%%%%%%%%
\clearpage
\phantomsection
\section{On a Lagrangian submanifold of \texorpdfstring{$\boldsymbol{Z}$}{Z} and singular cubic surfaces}\label{sec:variety-P}

In Section~\ref{SubSec:Z} we recall the construction of the hyperk\"ahler variety $Z=Z(Y)$ associated with any smooth cubic fourfold 
$Y$ not containing a plane, often referred to as the LLSvS variety, together with some of its basic properties.
It carries the action of an antisymplectic involution which has been studied in \cite{antisymplectic, FMSOG-II}, whose fixed locus consists of two connected components, one isomorphic to the cubic fourfold, which we hence denote by $Y$, and a second one which we denote by $W$. The goal of this section is to study the less-understood variety $W$. In order to do that,  in Section~\ref{SubSec:P} we define a subvariety $P$ of $F\times F$ which is related to $W$ via a rational map $\varphi\colon F\times F\dashrightarrow Z$ introduced by Voisin which we recall in Section~\ref{SubSec:Z}.
A key tool in our strategy is to reduce statements to the study of cubic surfaces $S$ which are given by intersecting the cubic~$Y$ with a projective space $\mathbb P^3$ and consider lines and twisted cubics on $S$. We do so in Section~\ref{SubSec:Srfs}.
This analysis provides the key tools to prove the main result of the section, Theorem~\ref{thm: P->W}.

\subsection{The LLSvS variety $\boldsymbol{Z}$}\label{SubSec:Z}
Given a smooth cubic fourfold $Y$ not containing any plane, Lehn, Lehn, Sorger, and van Straten constructed in \cite{LLSvS} a hyperk\"ahler variety $Z$ of dimension 8 which parametrises families of twisted cubics on $Y$ and their flat degenerations. The LLSvS manifold is deformation equivalent to the Hilbert scheme of four points on a K3 surface; see \cite{AL,Lehn}.

As any twisted cubic has $3t+1$ as its Hilbert polynomial, the starting point is to consider the compactification $\Hilb^{\gtc}(Y)$
of the space of twisted cubics inside the Hilbert scheme $\Hilb^{3t+1}(Y)$. Curves parametrised by this variety are called \textit{generalised twisted cubics}. The variety $\Hilb^{\gtc}(Y)$ is smooth, of dimension 10, and admits a morphism to the Grassmannian $G:=\Gr(4,6)$, which maps any generalised twisted cubic to its linear span; see \cite[Theorem~A]{LLSvS}. 
This morphism factors through a $\P^2$-fibration onto a variety $Z'$ followed by a generically finite morphism $a'\colon Z'\to G$.
The hyperk\"ahler projective manifold $Z$ is obtained from $Z'$ by contracting the divisor parametrising linear systems of curves which are not arithmetically Cohen--Macaulay. The image of the divisor under the contraction is a Lagrangian subvariety isomorphic to the cubic fourfold~$Y$. The variety $Z$ comes with a generically finite rational map $a\colon Z \dashrightarrow G$ and thus a natural polarisation $\mathcal{L}$ induced by the Pl\"ucker polarisation of the Grassmannian $G$; see \cite[Theorem~B]{LLSvS}.
The family parametrised by cubic fourfolds form a projective locally complete family of polarised hyperk\"ahler manifolds with polarisation of degree 2 and divisibility 2; see \cite{AG, LPZ}.

The hyperk\"ahler manifold $Z$ has a biregular antisymplectic involution $\tau\colon Z \to Z$, as pointed out in \cite{Lehn-oberwolfach}. It can be defined as follows: let $C\subset Y$ be a generalised twisted cubic, and let $Q$ be any quadric (surface) in the linear span $\langle C \rangle$, which contains $C$. The intersection $Q\cap Y$ is a sextic curve consisting of $C$ and the residual generalised twisted cubic $C'$. The involution $\tau$ is defined by  mapping $C$ to $C'$.
Antisymplectic involutions of K3$^{[n]}$-type manifolds and this particular example have been studied in \cite{antisymplectic, FMSOG-II}.

\begin{theorem}[\textit{cf.} {\cite[Theorem 1.4]{FMSOG-II}}]
The fixed locus of the antisymplectic involution has two connected components, one of which is the Lagrangian subvariety $Y$.
The other component $W$ is a smooth fourfold of general type; in particular, $\omega_{W} = 3 \mathcal{L}|_{W}$, where $\mathcal L$ is the primitive polarisation on $Z$.
\end{theorem}

Voisin constructed a degree 6 rational map $\varphi\colon F\times F \dashrightarrow Z$ in the following way; see \cite[Theorem~4.8]{Voisin-map-varphi}. 
Given a general pair of skew lines $(L_1, L_2)$ and a point $p\in L_1$,  the union of $L_1$ and the conic residual to $L_2$ in the intersection $\langle L_2,\ p\rangle \cap Y$ is a generalised twisted cubic $C$. The map $\varphi$ is defined by the association $(L_1, L_2)\mapsto [C]$. The indeterminacy locus, its resolution, and its resolution's branch locus have been studied in \cite{Muratore, chen, mio}.
 
The involution $\tau$ fits in the following commutative diagram:
\[\xymatrix{
F\x F \ar[d]^\sigma \ar@{-->}[r]^{\varphi} & Z\ar[d]^\tau\\
F\x F \ar@{-->}[r]^{\varphi} & Z\rlap{,}
}
\]
where $\sigma\colon (L_1,L_2)\mapsto (L_2,L_1)$ switches the factors.

\subsection{A variety dominating $\boldsymbol{W}$}\label{SubSec:P}

We denote by $P$ the closure in $F\times F$ of 
\[
\{ (L_1,L_2)\in F\times F \mid L_i \text{ are of type I, $L_1\not = L_2$, and } \psi(L_1)=\psi(L_2) \}.
\]
The variety $P$ turns out to have a central role in our study: the following intriguing lemma is the starting point which allows us to relate properties of $\psi$ and $\phi$.

\begin{lemma}\label{lem: fundamental}
 Let $(L_1, L_2)\in P$ be a general pair of lines, satisfying the conditions
\begin{enumerate}
    \item $L_1$ and $L_2$ are skew,
    \item $L_1$ and $L_2$ are both of type I,  
    \item and $\psi(L_1) = \psi(L_2)$.
\end{enumerate}
    Then the linear span of\, $L_1$ and $L_2$ intersects $Y$ in a cubic surface with four singular points of type $A_1$.
\end{lemma}

Cubic surfaces with four singularities of type $A_1$ are called Cayley cubics;  we analyse them in Section~\ref{SubSec:Srfs}.

\begin{proof}
We claim that the linear span $\Lambda:=\langle L_1, L_2\rangle$ intersects $Y$ in a cubic surface with two singular points on $L_1$. We denote by $R$ the residual line $\psi(L_1)$; then $R$ and $L_1$ are coplanar, and we consider the line $T$ in $\P(V^*)$ of hyperplanes which contain $\Lambda$.

As both $T$ and $\scrG (L_1)$ lie in the 2-dimensional linear space $\langle L_1, R \rangle^* \subset \P(V^*)$, the line $T$ intersects the smooth conic $\scrG (L_1)$ in two points, which correspond to the image under the Gauss map of two singular points of $S:=\Lambda\cap Y$ lying on the line $L_1$.

Repeating the argument for $L_2$, we find two singular points of $S$ on $L_2$, and we deduce that the cubic surface $S$ is Cayley by \cite[Corollary~9.2.3]{CAG}.
\end{proof}

Following the construction of the Voisin map, one gets the following. 

\begin{lemma}\label{fundamental lemma}
Let $L_1$ and $L_2$ be two lines that satisfy the following general conditions:
\begin{enumerate}\label{general-conditions}
    \item $L_1$ and $L_2$ are of type I,
    \item $L_1$ and $L_2$ are skew.
\end{enumerate}
If\, $\psi(L_1)= \psi(L_2)=:R$, then $\varphi(L_1, L_2) = \varphi(L_2, L_1)$.
\end{lemma}

\begin{proof}
By hypothesis we have
$\Lambda_{L_1}\cap Y = 2L_1 + R$ and $
\Lambda_{L_2}\cap Y = 2L_2 + R$.
Since \(L_1\) and \(L_2\) are skew, each meets \(R\) in a single point, say 
\(L_i \cap R = \{p_i\}\) for \(i=1,2\).
We compute \(\varphi(L_1,L_2)\) using the point \(p_1\), as described in
Section~\ref{SubSec:Z}.
Because the construction there applies to a general point of \(L_1\),
we must extend it to this special point \(p_1\).

The plane \(\langle p_1, L_2\rangle\) coincides with \(\Lambda_{L_2}\); 
hence its intersection with \(Y\) is the cubic curve \(2L_2 + R\). The union \(L_2 \cup R\) is the  conic residual to \(L_2\) in this cubic:
this agrees with the definition of residual conic obtained by taking the flat limit.
Indeed, set \(S := Y \cap \langle L_1, L_2\rangle\), and consider the morphism $
L_1 \rightarrow \Hilb^{2t+1}(S)$ mapping a point $P$ of $L_1$ to the conic in $\langle P, L_2\rangle \cap S$ residual to $L_2$.
Then the limiting residual conic corresponding to \(p_1\) is precisely \(L_2 \cup R\).
Finally, one checks that \(L_1 \cup L_2 \cup R\) is a point in \(\Hilb^{3t+1}(S)\), so 
$\varphi(L_1,L_2) = L_1 \cup R \cup L_2$.
The analogous computation for \(\varphi(L_2,L_1)\) yields the same conclusion,
proving the claim.
\end{proof}

We report the following theorem, which is a key result for our study, but we defer its proof to later, because it relies on some observations concerning Cayley cubics.
We recall  from Section~\ref{SubSec:Z} that the variety $Z'$ admits a morphism $a'$ onto the Grassmannian $G=\Gr(4,6)$, which descends to a rational map $a\colon Z\dashrightarrow G$ whose locus of indeterminacy is the connected component of the fixed locus of the involution $\tau$ isomorphic to $Y$.

\begin{theorem}\label{thm: P->W}
The Voisin map $\varphi\colon F\times F \dashrightarrow Z$ restricts to a dominant rational map $\varphi|_P\colon P\dashrightarrow W$. 
    
    The regular map $a|_{W}\colon W\to G$ is birational onto the image, which consists of the projective spaces that intersect~$Y$ in a Cayley cubic or a degeneration of it.
\end{theorem}

In the following we analyse the maps $P\dashrightarrow W \to G$ in greater detail. In order to do so, we describe singular cubic surfaces, focusing on lines on them.

\subsection{Cayley cubics and their degenerations}\label{SubSec:Srfs}
The stratification of the space of cubic surfaces given by the type of singularities has been studied by Bruce in \cite{bruce}. Here we are interested only in cubic surfaces with exactly four singular points of type $A_1$, called Cayley cubics, and their degenerations. From Bruce's analysis, degenerations of Cayley cubics can be summed up by the following diagram:
\begin{align}\label{degeneration-surfaces}
4A_1\ \succeq\ 2A_1A_3\ \succeq\ A_1A_5\ \succeq\ \wt E_6\ \succeq X_6\ \succeq X_7\ \succeq X_8\ \succeq\ X_9, 
\end{align}
where $\wt E_6$ denotes the cubics with one simple elliptic singularity and the $X_i$ are non-normal cubic surfaces. We say that a cubic surface~$S$ is of type $2A_1 A_3$ if $S$ has exactly two singularities of type $A_1$ and one singularity of type $A_3$, and analogously for the other singularities. 

In the following we focus on lines and twisted cubics on singular cubic surfaces to study the maps~$\varphi$ and~$\psi$. 
We notice that there is not an analogue of the map $\psi$ for the variety of lines on a cubic surface~$S$ because a line $L$ need not
have a plane tangent at every point. However, if $\P^2\simeq \pi\subset \langle S\rangle$ is tangent at every point of $L\subset S$ with $\pi \cap S = 2L + R$ for a line $R$, we say that the line $R$ is \textit{residual} and that $L$ is \textit{triple} if $R=L$.
Generalised twisted cubics on singular cubic surfaces have already been carefully analysed in \cite[Section~2]{LLSvS}, which we closely follow.
Every cubic surface $S$ with at most rational double points has a minimal resolution $\wt S\to S$ with $\wt S$ isomorphic to the blow-up of $\mathbb P^2$ in six points.
The orthogonal $K_{\wt S}^{\perp}\subset H^2(\wt S, \mathbb Z)$ is isomorphic to the negative-definite root lattice $E_6$. We denote by $\Gamma$ the roots in $K_{\wt S}^{\perp}$ and by $\Gamma_0$ the subset of effective roots. The latter correspond to irreducible components of exceptional curves for the resolution. We also denote by $W(\Gamma_0)$ the subgroup of the Weyl group of $E_6$ generated by reflections in the effective roots~$\Gamma_0$.
The moduli space of generalised twisted cubics on $S$ with reduced structure is then given by a union of projective planes; see \cite[Theorem~2.1]{LLSvS}:
\begin{align}\label{Hilb-S}
\Hilb^{\gtc}(S)_{\red} = (\Gamma/W(\Gamma_0)) \times \P^2.
\end{align}
In particular, linear systems of generalised twisted cubics on $S$ are in bijection with the orbits of $\Gamma$ under the action of $W(\Gamma_0)$. Along this bijection, families of arithmetically Cohen--Macaulay (aCM) curves correspond to orbits containing an effective root.

Going again through the construction  of $Z$ from \cite{LLSvS}, let $Y$ be a cubic fourfold not containing a plane. Mapping a generalised twisted cubic to its linear span defines a morphism from $\Hilb^{\gtc}(Y)$ to the Grassmannian $G=\Gr(4,6)$ which factors through a $\mathbb P^2$-fibration followed by a generically finite morphism $a'\colon Z'\to G$. Let $C$ be a twisted cubic and $E$ be its linear span; the intersection $S:=Y\cap E$ is a cubic surface.
If $S$ has ADE singularities, the preimage $Z'(S):=(a')^{-1}(E)$ is in bijection with the set $\Gamma/W(\Gamma_0)$ in Equation~\eqref{Hilb-S} and the points $Z'_{\aCM}(S)$ corresponding to aCM curves  are in bijection with $(\Gamma\setminus \Gamma_0)/W(\Gamma_0)$. By its definition the involution $\tau$ restricts to an action on $Z'_{\aCM}(S)$ and is given by $\alpha \mapsto -\alpha$ for $\alpha \in \Gamma$; see \cite{Lehn-oberwolfach}.

We defer to the appendix details about these singular surfaces; here we recall only the main facts we need to analyse the Voisin maps.

We also introduce a piece of notation. We denote by $G_{4A1}$ the closure in the Grassmannian $G$ of 
\[
\{ E\in G \mid E\cap Y \mbox{ is a  singular cubic surface of type }4A_1 \}.
\]
 Analogously, we define the loci $G_{2A_1A_3}$, $G_{A_1A_5}$, $G_{\wt E_6}$, and for the non-normal surfaces, $G_{X_6}$, $G_{X_7}$, $G_{X_8}$, $G_{X_9}$.
\bigskip

\subsubsection*{Singular cubic surfaces of type $\boldsymbol{4A_1}$}
Any Cayley cubic, \textit{i.e.}~a cubic surface with exactly four $A_1$ singular points, contains exactly nine lines whose schematic representation resembles a tetrahedron with vertices at the singular points. We name \textit{edges} the six of them connecting two singular points. In particular, if $P_0$, $P_1$, $P_2$, $P_3$ are the singular points, we denote $E_{ij}$ the line containing $P_i$ and $P_j$. Each of the other three lines intersects exactly two opposite edges; we denote by $J_{ij,kl}$ the line joining the edges $E_{ij}$ and $E_{kl}$ for $\{i,j,k,l\} = \{0,1,2,3\}$. See Figure~\ref{fig:4A1}.

\begin{figure}[h!]
    \centering
    \captionsetup[subfigure]{justification=centering}
    \begin{subfigure}[t]{0.52\textwidth}
        \centering
\begin{tikzpicture}[x=0.75pt,y=0.75pt,yscale=-1,xscale=1]
%uncomment if require: \path (0,310); %set diagram left start at 0, and has height of 310

%Shape: Triangle [id:dp4305268311793562] 
\draw   (313,38.47) -- (450.5,262.04) -- (175.5,262.04) -- cycle ;
%Straight Lines [id:da5726588494861005] 
\draw    (313,38.47) -- (317.03,178.12) ;
%Straight Lines [id:da1536553352510559] 
\draw    (317.03,178.12) -- (175.5,262.04) ;
%Straight Lines [id:da6119140812049109] 
\draw    (317.03,178.12) -- (450.5,262.04) ;
%Shape: Arc [id:dp8345163812779048] 
\draw  [draw opacity=0] (269.15,197.77) .. controls (268.26,193.46) and (267.8,189.03) .. (267.8,184.5) .. controls (267.8,155.17) and (287.19,129.69) .. (315.66,116.85) -- (357.65,184.5) -- cycle ; \draw  [color={rgb, 255:red, 208; green, 2; blue, 27 }  ,draw opacity=1 ] (269.15,197.77) .. controls (268.26,193.46) and (267.8,189.03) .. (267.8,184.5) .. controls (267.8,155.17) and (287.19,129.69) .. (315.66,116.85) ;  
%Shape: Arc [id:dp8011286053870714] 
\draw  [draw opacity=0] (356.5,262) .. controls (356.5,262) and (356.5,262) .. (356.5,262) .. controls (317.84,262) and (285.01,240.33) .. (273.19,210.21) -- (356.5,185.5) -- cycle ; \draw  [color={rgb, 255:red, 208; green, 2; blue, 27 }  ,draw opacity=1 ] (356.5,262) .. controls (356.5,262) and (356.5,262) .. (356.5,262) .. controls (317.84,262) and (285.01,240.33) .. (273.19,210.21) ;  
%Shape: Arc [id:dp1820879557587991] 
\draw  [draw opacity=0] (361.11,117.19) .. controls (371.99,142.82) and (368.74,171.71) .. (353.78,193.75) -- (287.72,147.9) -- cycle ; \draw  [color={rgb, 255:red, 5; green, 26; blue, 245 }  ,draw opacity=1 ] (361.11,117.19) .. controls (371.99,142.82) and (368.74,171.71) .. (353.78,193.75) ;  
%Shape: Arc [id:dp9166531713834731] 
\draw  [draw opacity=0] (347.39,202.4) .. controls (342.83,207.39) and (337.52,211.81) .. (331.48,215.48) .. controls (307.03,230.36) and (277.09,229.15) .. (253.14,214.97) -- (289.6,146.64) -- cycle ; \draw  [color={rgb, 255:red, 5; green, 26; blue, 245 }  ,draw opacity=1 ] (347.39,202.4) .. controls (342.83,207.39) and (337.52,211.81) .. (331.48,215.48) .. controls (307.03,230.36) and (277.09,229.15) .. (253.14,214.97) ;  
%Shape: Arc [id:dp9316643160338189] 
\draw  [draw opacity=0] (322.73,130.27) .. controls (327,131.57) and (331.13,133.39) .. (335.04,135.74) .. controls (357.76,149.43) and (367.69,177.83) .. (363.55,208.66) -- (285.92,217.34) -- cycle ; \draw  [color={rgb, 255:red, 126; green, 211; blue, 33 }  ,draw opacity=1 ] (322.73,130.27) .. controls (327,131.57) and (331.13,133.39) .. (335.04,135.74) .. controls (357.76,149.43) and (367.69,177.83) .. (363.55,208.66) ;  
%Shape: Arc [id:dp10604063744649128] 
\draw  [draw opacity=0] (238.81,158.65) .. controls (259.05,137.86) and (285.3,127.04) .. (309.5,129.02) -- (288.16,218.74) -- cycle ; \draw  [color={rgb, 255:red, 126; green, 211; blue, 33 }  ,draw opacity=1 ] (238.81,158.65) .. controls (259.05,137.86) and (285.3,127.04) .. (309.5,129.02) ;  

% Text Node
\draw (308.75,266.88) node [anchor=north west][inner sep=0.75pt]   [align=left] {E\textsubscript{12}};
% Text Node
\draw (208.67,135.62) node [anchor=north west][inner sep=0.75pt]   [align=left] {E\textsubscript{02}};
% Text Node
\draw (395.13,128.85) node [anchor=north west][inner sep=0.75pt]   [align=left] {E\textsubscript{01}};
% Text Node
\draw (291.27,89.62) node [anchor=north west][inner sep=0.75pt]   [align=left] {E\textsubscript{03}};
% Text Node
\draw (358.87,220.02) node [anchor=north west][inner sep=0.75pt]   [align=left] {E\textsubscript{13}};
% Text Node
\draw (221.47,233.42) node [anchor=north west][inner sep=0.75pt]   [align=left] {E\textsubscript{23}};
% Text Node
\draw (272,158.2) node [anchor=north west][inner sep=0.75pt]  [font=\footnotesize,xslant=0]  {$J_{03,12}$};
% Text Node
\draw (298.4,206.6) node [anchor=north west][inner sep=0.75pt]  [font=\footnotesize]  {$J_{01,23}$};
% Text Node
\draw (330.4,154) node [anchor=north west][inner sep=0.75pt]  [font=\footnotesize]  {$J_{02,13}$};
% Text Node
\draw (314,15.8) node [anchor=north west][inner sep=0.75pt]    {$P_{0}$};
% Text Node
\draw (459.2,257.4) node [anchor=north west][inner sep=0.75pt]    {$P_{1}$};
% Text Node
\draw (147.6,256.6) node [anchor=north west][inner sep=0.75pt]    {$P_{2}$};
\end{tikzpicture}
\caption{Lines on a singular cubic surface of type $4A_1$}
        \label{fig:4A1}
    \end{subfigure}%
    \hfill
    \begin{subfigure}[t]{0.45\textwidth}
        \centering
\begin{tikzpicture}[x=0.75pt,y=0.75pt,yscale=-1,xscale=1]
%uncomment if require: \path (0,310); %set diagram left start at 0, and has height of 310

%Straight Lines [id:da3077218017664136] 
\draw    (279,63) -- (356,204) ;
%Straight Lines [id:da20910350589996107] 
\draw    (212,206) -- (279,63) ;
%Straight Lines [id:da6079596637848914] 
\draw    (212,206) -- (356,204) ;
%Shape: Arc [id:dp2474339554067404] 
\draw  [draw opacity=0][line width=1.5]  (260.19,24.83) .. controls (270.77,40.35) and (280.91,66.22) .. (287.21,96.41) .. controls (294.55,131.62) and (294.76,163.54) .. (288.93,181.22) -- (257.84,102.54) -- cycle ;
\draw  [color={rgb, 255:red, 10; green, 15; blue, 219 }  ,draw opacity=1 ][line width=1.5]  (260.19,24.83) .. controls (270.77,40.35) and (280.91,66.22) .. (287.21,96.41) .. controls (294.55,131.62) and (294.76,163.54) .. (288.93,181.22) ;  
%Shape: Arc [id:dp8333253139923811] 
\draw  [draw opacity=0][line width=1.5]  (273.85,127.49) .. controls (287.42,135.81) and (300.95,164.16) .. (306.8,199.29) .. controls (309.84,217.54) and (310.32,234.74) .. (308.66,248.77) -- (277.21,204.21) -- cycle ;
\draw  [color={rgb, 255:red, 9; green, 9; blue, 207 }  ,draw opacity=1 ][line width=1.5]  (273.85,127.49) .. controls (287.42,135.81) and (300.95,164.16) .. (306.8,199.29) .. controls (309.84,217.54) and (310.32,234.74) .. (308.66,248.77) ;  

% Text Node
\draw (272,14) node [anchor=north west][inner sep=0.75pt]   [align=left] {L\textsubscript{1}};
% Text Node
\draw (314,237) node [anchor=north west][inner sep=0.75pt]   [align=left] {L\textsubscript{2}};
% Text Node
\draw (244,211.4) node [anchor=north west][inner sep=0.75pt]    {$R_{03}$};
% Text Node
\draw (208,125) node [anchor=north west][inner sep=0.75pt]   [align=left] {R\textsubscript{02}};
% Text Node
\draw (325,119) node [anchor=north west][inner sep=0.75pt]   [align=left] {R\textsubscript{23}};
% Text Node
\draw (288,50) node [anchor=north west][inner sep=0.75pt]   [align=left] {P\textsubscript{2}};
% Text Node
\draw (187,194) node [anchor=north west][inner sep=0.75pt]   [align=left] {P\textsubscript{0}};
% Text Node
\draw (360,197) node [anchor=north west][inner sep=0.75pt]   [align=left] {P\textsubscript{3}};
% Text Node
\draw (277,147) node [anchor=north west][inner sep=0.75pt]   [align=left] {};%{P\textsubscript{1}};
\end{tikzpicture}
\caption{Lines on a $2A_1A_3$ surface}
        \label{fig:2A1A3}
    \end{subfigure}
    \caption{}
    \label{fig:figures}
\end{figure}

We denote by $K$ the linear system of generalised twisted cubics containing the reducible elements given by the lines through any singular point $P_i$: 
 \[
 K = [E_{ij}\cup E_{ik} \cup E_{il}] \in Z'_{aCM}(S) 
 \]
 for $\{i,j,k,l\}=\{0,1,2,3\}$.
 
\begin{lemma}\label{lem: opposite edges}
   A pair of opposite edges maps to $K$ under $\varphi$; that is, 
   \begin{align*}
       \varphi(E_{ij}, E_{kl}) = K\quad \text{ for indices such that }  \{i,j,k,l\}=\{1,2,3,4\}.
   \end{align*}
\end{lemma}

\begin{proof}
    See the appendix.
\end{proof}

\begin{lemma}\label{lem: 4A_1 LLSvS}
Let $S$ be a Cayley cubic. Then
\begin{align*}
    \# Z'_{aCM}(S)=13\quad \text{and}\quad \#Z_{aCM}'(S)^{\tau} = 1,
\end{align*}
where $Z_{aCM}'(S)^{\tau}$ denotes the locus of fixed points of the involution $\tau$.
\end{lemma}

\begin{proof}
    The first equality follows from \cite[Table~1]{LLSvS}. The second one can be deduced from the description of $\tau$ in terms of roots.
\end{proof}
\medskip

\begin{proof}[Proof of Theorem~\ref{thm: P->W}]
    By Lemma~\ref{lem: fundamental} the general point $(L_1,L_2)$ in $P$ corresponds to opposite edges on a Cayley cubic surface. By Lemma~\ref{lem: opposite edges} its image $\varphi((L_1, L_2))$ is an aCM generalised twisted cubic fixed by the involution on $Z$, \textit{i.e.}~a point in $W$. In particular, $\varphi$ restricts to a dominant map $\varphi: P\dashrightarrow W$.
    Finally, the morphism $a|_W\colon W\to G$ is generically 1:1 onto its image because, by Lemma~\ref{lem: 4A_1 LLSvS}, given a projective space $\Lambda \simeq \mathbb P^3$ which intersects $Y$ in a Cayley cubic $S$, $K$ is the unique linear system of $aCM$ generalised twisted cubics in $S$ which is fixed by $\tau$.
\end{proof}

\subsubsection*{Singular cubic surfaces of type $\boldsymbol{2A_1A_3}$}

Let $S$ be a cubic surface with exactly one singular point $P_2$ of type $A_3$ and two singular points $P_0$, $P_3$ of type $A_1$. It contains exactly five lines as depicted in Figure~\ref{fig:2A1A3}: the line $R_{02}$ through $P_0$, $P_2$; the line $R_{23}$ through $P_2$, $P_3$; the line $R_{03}$ through $P_{0}$, $P_3$; the line $L_1$ containing the only singular point $P_2$; and the line $L_2$ contained in the smooth locus.

\begin{proposition}\label{prop:2A1A3-residual}
Let $S$ be a cubic surface of type $2A_1A_3$. Then 
    \begin{align*}
        &L_1\mbox{ is residual to }R_{23},\quad
         L_1 \mbox{ is residual to }R_{02} ,\\         
       & L_2\mbox{ is residual to }R_{03},\quad
        L_2\mbox{ is residual to } L_1.
    \end{align*}
    
Moreover, there is no plane in $\langle S\rangle\simeq \mathbb P^3$ tangent to $S$ at every point of\, $L_2$.
\end{proposition}

\begin{proof}
    See the appendix.
\end{proof}

\begin{lemma}
Let $S$ be a cubic surface of type $2A_1A_3$. Then
\begin{align*}
    \# Z'_{aCM}(S)=5\quad \mbox{and}\quad \#Z_{aCM}'(S)^{\tau} = 1.
\end{align*}
\end{lemma}

\begin{proof}
    The first equality follows from \cite[Table~1]{LLSvS}. The second one can be deduced from the description of $\tau$ in terms of roots.
\end{proof}

\subsubsection*{Singular cubic surfaces of type $\boldsymbol{A_1A_5}$}\label{ssec:A1A5}

Let $S$ be a cubic surface with an $A_1$ and an $A_5$ singularity. It contains exactly two lines, say $L_1$ passing through both singular points and $L_2$ which intersects the singular locus in the $A_5$ point.

\begin{proposition}\label{prop: lines-A1A5}
    Let $S$ be a cubic surface of type $A_1A_5$. Then
    \begin{align*}
        L_2 \text{ is residual to itself\, $($i.e.\ triple on $S)$, and } L_2 \text{ is residual to } L_1.
    \end{align*}
\end{proposition}

\begin{proof}
  See the appendix.
\end{proof}
  
\begin{lemma}\label{lem:gtc-on-A1A5}
Let $S$ be a cubic surface of type $A_1A_5$. Then
\begin{align*}
    \# Z'_{aCM}(S)= \#Z_{aCM}'(S)^{\tau} = 1.
\end{align*}
\end{lemma}

\begin{proof}
    This is clear from \cite[Table~1]{LLSvS}.
\end{proof}

\subsubsection{Simple elliptic cubic surfaces}
Any cubic surface $S$ with a simple elliptic singularity is a cone over a planar elliptic curve $E$. Let $v$ be the vertex. The Fano variety $F(S)$ of $S$ can be naturally identified with the embedded elliptic curve $E$; indeed, to any point $P$ of $E$ corresponds the line through $P$ and $v$.
The scheme of generalised twisted cubics is isomorphic to $\Sym^3(E)$, and the summation map to $E$ is a $\P^2$-bundle giving the fibration $\Hilb^{\gtc}(S)\to Z'(S)$; see \cite[Proposition~2.7]{LLSvS}. An aCM generalised twisted cubic corresponds to the union of three lines, which join the vertex with three non-collinear points on $E$. If the three points are collinear, then they are the support of a non-aCM curve with an embedded point at the vertex. Via the isomorphism  $F(S)\simeq E$, the residual line can be expressed in terms of the group law of  $E$, and the analogue of $P$ for lines on the surface $S$, 
\begin{align*}
    P(S):= \{(p_1,p_2)\in F(S)\times F(S) \mid \mbox{ $p_1$ and $p_2$ have the same residual}\}
\end{align*}
admits a nice description, as follows. 

\begin{proposition}\label{prop:triplesOnE6}
    Let $S$ be a simple elliptic cubic surface. Then
    \begin{itemize}
        \item the residual to the line $p\in E\simeq F(S)$ is given by $-2p$, 
        \item $P(S) =\{(p,p+\xi) \mid  \xi\in E[2]\}$, 
        \item the triple lines on $S$ are given by the $3$-torsion points $E[3]\subset E$; in particular, every simple elliptic cubic surface has nine triple lines.
    \end{itemize}
\end{proposition}

We now explore  some connections of our study of singular cubic surfaces with results of Gounelas and Kouvidakis \cite{GK-lines}. As a consequence, we deduce the existence of a surface of type $A_1A_5$, which is a fundamental little ingredient for the main theorem.

Let $Y$ be a smooth cubic fourfold. The natural map $\gamma\colon \mathbb L \to Y$ from the universal family of lines to the cubic fourfold is a fibration in (2,3)-complete intersection curves. Given a point $y\in Y$, its fibre $\gamma^{-1}(y)$ consists of the lines through the point $y$. Gounelas and Kouvidakis studied the $g^1_3$ of these curves coming from the rulings of the corresponding quadric and their relation with the Voisin map.
 They proved that for the general cubic fourfolds, the surface $V$ of triple lines intersects $S_{\II}$ in an irreducible curve with 3780 nodes. In particular, the general triple line is of type $I$; see \cite[Theorem~B]{GK-lines}.

\begin{corollary}\label{cor:lineSecTyp}
    Let $Y$ be a general cubic fourfold.
\begin{enumerate}
    \item\label{c:lST-1} 
    For any $L\in S_{\II}$, the intersection $\Lambda_L\cap Y$ is a non-normal singular cubic surface.
    \item\label{c:lST-2}  If\, $L\in V\cap S_{\II}$, then $\Lambda_L\cap Y$ is of type $X_7$ or worse.
    \item\label{c:lST-3}  If\, $L\in (V\cap S_{\II})_{\sing}$ is a node, then $\Lambda_L\cap Y$ is of type $X_8$ or worse. In particular, there are $3780$ projective spaces of dimension $3$ intersecting $Y$ in a singular surface of type $X_8$.
    \item\label{c:lST-4}  For any $L\in S_{\II}$, the surface $\Lambda_L\cap Y$ is not of type $X_9$.
\end{enumerate}
\end{corollary}

\begin{proof}
    Given a line $L$ of type $\II$, the cubic surface $\Lambda_L \cap Y$ is singular along the entire $L$.
    By the study of Gounelas and Kouvidakis, a line as in item~\eqref{c:lST-2} corresponds to a triple line of type II, and one as in item~\eqref{c:lST-3} corresponds to a type II line with two triple tangent 2-planes; see \cite[Comments right above Remark~4.20]{GK-lines}.
    The distinction in cases comes from an elementary analysis of lines on non-normal cubic surfaces (see the appendix). The second statement in item~\eqref{c:lST-3} follows from \cite[Theorem~B]{GK-lines}.

    For the last item, let $L$ be a line of type $\II$ such that $\Lambda_L\cap Y$ is a singular cubic surface of type $X_9$, that is, a cone over a cuspidal cubic curve. Then the lines passing through the vertex of the cone form a curve, which contains a cuspidal cubic. This is in contradiction with 
    \cite[Proposition 3.5]{GK-lines}.
\end{proof}

\begin{proposition}\label{cor: A1A5-E6}
    Let $Y$ be a general cubic fourfold, and let $L_1$ and $L_2$ be general distinct lines such that $L_2$ is residual to $L_1$ and $L_2$ is of type $I$ and triple. Then $\langle {L_1}, \Lambda_{L_2} \rangle$ is a $3$-dimensional space which cuts $Y$ in a cubic surface of type either $A_1A_5$ or $\wt E_6$.
\end{proposition}

\begin{proof}
    By Lemma~\ref{lem: fundamental} a general point of $P$ consists of a pair of lines whose linear span intersects $Y$ in a Cayley cubic. Thus, if $L_1$ and $L_2$ are as in the hypothesis, then the 3-dimensional linear space $\langle L_1, \Lambda_{L_2}\rangle$ intersects $Y$ in a cubic surface $S$ which is a degeneration of a Cayley cubic and contains $L_2$ as a triple line on the surface. By the analysis of lines on cubic surfaces, $S$ is either of type $A_1A_5$, of type $\wt E_6$, or non-normal.
 As both $V$ and $S_{\II}$ are surfaces, both the families of non-normal cubic surfaces and of surfaces containing a triple line, which is triple on the surface, have dimension 2. By Corollary~\ref{cor:lineSecTyp} and the appendix, the general non-normal cubic surface (which is a linear section of $Y$) does not contain a triple line; thus the general surface containing a triple line is of type $\widetilde E_6$ or $A_1 A_5$.
\end{proof}

\begin{corollary}\label{cor:existence A1A5}
    Let $Y$ be a general cubic fourfold. Then  there exists a $\Lambda \in G$ such that $\Lambda\cap Y$ is a cubic surface of type $A_1A_5$.
\end{corollary}

\begin{proof}
Let $L_1$, $L_2$ be general distinct lines such that $L_2$ is a triple line of type $I$ and $L_2$ is residual to $L_1$. 
We consider the projective space $\Lambda:=\langle \Lambda_{L_2}, L_1\rangle$ and the singular cubic surface $S:=\Lambda\cap Y$ it defines. In particular, the line $L_2$ is triple on $S$ and is residual to $L_1$. By Proposition~\ref{cor: A1A5-E6} the surface $S$ is either singular of type $A_1A_5$ or isomorphic to a cone over an elliptic curve.

In other words, the association $(L,R)\mapsto \Lambda$ for pairs of lines as above defines a rational map $\widehat\psi^{-1}(V)\dashrightarrow G$ whose image lies in the union of loci $G_{A_1A_5}\cup G_{\widetilde E_6}$. This is generically finite by Propositions~\ref{prop:triplesOnE6} and~\ref{prop: lines-A1A5}. Since $G_{\widetilde E_6}$ is 1-dimensional by the discussion in \cite[Section~1.4]{Nesterov-Oberdieck}, the locus $G_{A_1A_5}$ is not empty.
\end{proof}

\section{The monodromy group of \texorpdfstring{$\boldsymbol{\psi}$}{psi} is maximal}\label{sec:monodromy}
In this section we prove that the monodromy group of $\wh\psi$ is maximal. In order to do so, we use that the ramification of $\wh\psi$ is simple, as proven in Section~\ref{sec:ramification} and the results of Section~\ref{sec:variety-P} about the variety $P$.

Given a generically finite dominant morphism $f\colon X\to Y$ of degree $d$ between irreducible varieties (necessarily of the same dimension), we obtain a degree d extension of function fields $k(X)/k(Y)$, and taking the Galois closure $K/k(Y)$ of this extension, we set $\Gal_f = \Gal(K/k(Y))$. This agrees with the usual monodromy group $\Mon_f$ (see \cite{Har,sottile-yahl}), which is defined as the image in the symmetric group $S_d$ of the group of deck transformations of the unramified (\textit{i.e.}~topological) cover $X\setminus f^{-1}(\Branch(f)) \to U$, where $U := Y \setminus\Branch(f)$.
Recall the following classical results for $f\colon X\to Y$ and $U$ as above.

\begin{proposition}[\textit{cf.} {\cite[Section~II.3, p.~698]{Har}}]\label{prop:transposition}
    If there exists a fibre of $f$ with a unique point of ramification index $2$ and all other points unramified, then $\Mon_f \subset S_d$ contains a transposition.
\end{proposition}

Let $X^{(s)}_U$ be the complement of the big diagonal in the fibre product of $X_U$ $s$-times with itself over $U$. In other words, the fibre of the natural morphism $X^{(s)}_U \to U$ consists
of $s$ distinct points in the fibre of $f$.

\begin{lemma}[\textit{cf.} {\cite[Proposition 2]{sottile-yahl}}] \label{lem:2-transitive}
The variety $X^{(s)}_U$ is irreducible if and only if\, $\Mon_f$ is an s-transitive subgroup of $S_d$.
\end{lemma}

We now want to apply  these results to the resolution of the Voisin map $\widehat\psi\colon \widehat F\to F$.
For the remainder of the section, let $Y$ be a general cubic fourfold, $F$ be its Fano variety of lines, and $U\subset F$ be the complement of the branch locus of $\widehat \psi$.
We recall that $P\subset F\times F$, introduced in Section~\ref{SubSec:P}, is mapped onto $W$ under the map $\varphi\colon F\times F\dashrightarrow Z$.
Thanks to Proposition~\ref{cor:ram-simple}, we can readily apply Proposition~\ref{prop:transposition} and deduce that the monodromy group $\Mon_{\wh\psi}$ contains a transposition.
In order to apply Lemma~\ref{lem:2-transitive}, we are led to study the irreducibility of $\wh F^{(2)}_U$, which is strictly related to the variety $P$. The rest of the section is then devoted to proving the following, which is the crucial ingredient to conclude the maximality of monodromy.

\begin{proposition}\label{prop:irreducible}
    The variety $P$ is irreducible.
\end{proposition}

We postpone the proof and illustrate right away how to get the maximality of monodromy.

\begin{theorem}\label{thm:monodromy-maximal}
The monodromy of\, $\wh\psi$ is maximal; i.e.\ $\Mon_{\wh\psi} = S_{16}$. 
\end{theorem}

\begin{proof}
    From Propositions~\ref{cor:ram-simple} and~\ref{prop:transposition}, we deduce that $\Mon_{\wh\psi}$ contains a transposition, and by Proposition~\ref{prop:irreducible} and Lemma~\ref{lem:2-transitive}, the action of $\Mon_{\wh \psi}$ is $2$-transitive. Thus $\Mon_{\wh\psi}$ contains all transpositions: indeed, up to renumbering, we can assume the transposition is $(1,2)\in \Mon_{\wh \psi}$. By 2-transitivity, there exists a $\sigma\in\Mon_{\wh\psi}$ such that $\sigma(1)=a$ and $\sigma(2)=b$ for any $a$ and $b$ and we have
\[
\sigma \circ (1,2) \circ \sigma^{-1} = (a,b).
\]
Hence $\Mon_\wh\psi$ contains any transposition and coincides with the entire symmetric group on 16 elements.
\end{proof}

In order to prove the irreducibility of $P$, we take advantage of the map $\varphi \colon P \dashrightarrow W$ and use its various properties we have showed in Section~\ref{sec:variety-P} to compute the ramification of $\varphi$ at some special points. Clearly, there is an open subset $U_{\et}\subset W$ for which the base change $P_{U_{\et}} \to U_{\et}$ is \'etale of degree 6 and proper, but as we are interested in the ramification of $\varphi$, we need to study a slightly bigger open subset; thus we provide the following technical proposition.
Recall that $Z$ admits a generically finite map $a\colon Z\dashrightarrow G$ whose indeterminacy locus is $Y$.

\begin{proposition}\label{prop:flat-open}
There is an open $U_{\flatt}\subset W$ such that
\begin{enumerate}
    \item\label{p:fo-1} the pullback $\varphi\colon P_{U_\flatt} \to U_\flatt$ is a flat finite morphism, 
    \item\label{p:fo-2} there exists a point  $w\in U_\flatt$  such that the $3$-dimensional linear space $a(w)$ intersects $Y$ in a cubic surface of type $A_1A_5$.
\end{enumerate}
\end{proposition}

\begin{proof}
Let $P \xleftarrow{\pi}\wh P \xrightarrow{\wh\varphi} W$ be any resolution of indeterminacy of $\varphi$ and $\wh P \xrightarrow{\varphi_1} P_{\Stein} \xrightarrow{\varphi_2} W$ be its Stein factorization. 
We denote by $E_\pi$ and $E_1$ the exceptional loci of $\pi$ and $\varphi_1$, respectively. Then the morphism
\[
    \wh{P}\setminus \wh \varphi^{\; -1}\left(\wh \varphi \left( E_1\cup E_\pi \right)\right) \longrightarrow W \setminus \wh \varphi (E_1 \cup E_\pi) =:U
    \]
    is proper, as it is the base change of the proper morphism $\wh \varphi$. It has finite fibres, as we removed the exceptional divisor $E_1$. Thus it is a finite morphism. As we removed $E_\pi$, the source space is actually a subset of $P$.
    Hence, we have proven that the restriction of $\varphi$ to $P_U \to U$ is a finite morphism.

    Moreover, the further restriction of $\varphi$ to $U\setminus\varphi(P_{\sing})$ is a finite morphism of non-singular varieties, hence flat by \cite[Remark~4.3.11]{Liu}. This open set satisfies point~\eqref{p:fo-1} of the claim.\medskip
    
    We now construct explicitly such an open set for which we show point~\eqref{p:fo-2} of the claim. In order to do so, we consider the map $\varphi$ and the auxiliary map $\varphi_G$: 
\[
\begin{tikzcd}[row sep=tiny]
    P \ar[r,dashed, "\varphi"]\ar[rr, dashed, bend left, "\varphi_G"] & W\ar[r,"a"] & G.
\end{tikzcd}
\]
Their graph closures are going to serve as an explicit resolution of indeterminacy: 
    \begin{align*}
\wh P:=\closure\{((L_1,L_2),w)\in P\times W \mid \varphi(L_1,L_2) = w \},\\
\wh P_G:=\closure\{((L_1,L_2),g)\in P\times G \mid \varphi_G(L_1,L_2) = g \}. 
\end{align*}
As $\varphi_G = a\circ \varphi$, by continuity we have a natural map $\wh P \to \wh P_G$ given by
\begin{align}\label{map-P--PG}
\wh P \longrightarrow \wh P_G, \quad \left( L_1, L_2, w \right) \longmapsto \left( L_1, L_2, a(w) \right).
\end{align}
As before, we factor the resolution of indeterminacy $\wh \varphi \colon \wh P \to W$ in $\wh\varphi = \varphi_2\circ \varphi_1$ using Stein factorization, and we denote by $E_1$ the exceptional divisor of $\varphi_1$ and by $E_\pi$ the exceptional divisor of $\pi\colon \wh P \to P$.

By Corollary~\ref{cor:existence A1A5} there exists a point $w\in W$ such that $a(w)\in G$ is a point that intersects $Y$ in a surface~$S$ with singularities $A_1A_5$. By Proposition~\ref{prop: lines-A1A5} the surface $S$ contains exactly two lines $L_1$ and $L_2$, which by generality are of type $I$, such that $\psi(L_1) = L_2 = \psi (L_2)$. 
By the very definition of $\wh P_G$, its points consist of a pair of lines and a 3-dimensional linear subspace, which contains the two lines; hence 
\[
\wh\varphi_G^{\;-1} \left( a(w) \right) = \left\lbrace \left( \left( L_1,L_2 \right), a(w) \right),\ \left( \left( L_2,L_1 \right), a(w) \right) \right\rbrace.
\]
Now, as there is just one linear system of arithmetically Cohen--Macaulay generalised twisted cubics on a cubic surface of type $A_1A_5$ by Lemma~\ref{lem:gtc-on-A1A5}, we get
\[
\wh\varphi^{\;-1} \left( w \right) = \left\lbrace \left( \left( L_1,L_2 \right), w \right),\ \left( \left( L_2,L_1 \right), w \right) \right\rbrace.
\]
 As $\wh\varphi^{\;-1}(w)$ is finite, we have $w\not \in \wh\varphi (E_1)$.

Notice that $\varphi_G|_P\colon P\dashrightarrow G$ is defined by the association $(L_1,L_2)\mapsto \langle L_1,L_2\rangle = \langle L_1,L_2, R\rangle$, where $R = \psi(L_1) = \psi(L_2)$. Moreover, if both $L_1$ and $L_2$ are of type I, then we have $\varphi_G((L_1,L_2)) = \langle \Lambda_{L_1}, \Lambda_{L_2}\rangle$, hence pairs of distinct lines of type I in $P$ are not in the indeterminacy locus  of $\varphi$.
Since $Z_{aCM}(S)$ is a singleton for a cubic surface $S$ of type $A_1A_5$ by Lemma~\ref{lem:gtc-on-A1A5}, the map $\varphi$ is defined at the point $(L_1,L_2)$; thus we conclude that $w\not \in \wh\varphi(E_\pi)$.

Now we are left to prove that $(L_1, L_2)$ is a regular point of $P$ so that $\varphi$ is flat at the points $(L_1, L_2)$ and $(L_2, L_1)$ in $P$. We prove the regularity of $P$ at the point $(L_1,L_2)$ in the lemma below.
\end{proof}

\begin{lemma}\label{lem:singular-locus-P}
    The singular locus of\, $\wh F \times_{\wh\psi} \wh F$ is contained in $E \times_{\wh\psi} E$.
\end{lemma}

\begin{proof}
    Let $(f_1,f_2)\in \wh F \times_{\wh\psi} \wh F$ be a point with $\wh \psi(f_1) = \wh \psi(f_2) = R$. Then its tangent space is the fibre product over the diagram
    \[
\begin{tikzcd}[row sep=tiny]
    T_{f_1} \wh F \ar[rd, "d\wh\psi"] && T_{f_2} \wh F\rlap{,}  \ar[ld,"d\wh\psi", swap]\\
    & T_R F
\end{tikzcd}
\]
where $T_{f_1} \wh F$ and  $T_{f_2} \wh F$ have dimension 4 as $\wh F$ is a non-singular fourfold.
If one among $f_1$ and $f_2$ is not in the ramification of $\wh \psi$, which coincides with the exceptional divisor $E$, then $d\wh\psi$ is an isomorphism  at that point and the fibre product is a vector space of dimension 4, hence the claim.
\end{proof}

\begin{proof}[Proof of Proposition~\ref{prop:irreducible}]
First we are going to prove that $P$ modulo the switching involution $\sigma$ is irreducible. In order to do so, we start with a point $q\in U_{\flatt}\cap W_{A_1A_5}$, which exists by Proposition~\ref{prop:flat-open}. As exactly two lines lie on a surface with singularities $A_1A_5$ (see Proposition~\ref{prop: lines-A1A5}), the preimage $\varphi^{-1}(q)$ consists of two points, say
\[
(L_1,L_2) \quad \mbox{and} \quad (L_2, L_1).
\]
Using the involution $\sigma$ on $F\times F$ switching the factors, which clearly restrict to an involution of $P$, we see that $P$ at $(L_1, L_2)$ is locally isomorphic to $P$ at $\sigma((L_1,L_2)) = (L_2, L_1)$.
As $\varphi$ is flat, proper, and of degree~6 at the point $q$, we conclude that the schematic preimage $\varphi^{-1}(q)$ is a finite scheme with multiplicity 3 at both $(L_1,L_2)$ and $(L_2,L_1)$.
We now pick  a point $p\in U_{\et}$ close enough to $q$ and focus on the preimages $\varphi^{-1}(p) = \{p_1,\ldots, p_6\}$: if we let $p$ tend to $q$, then three points out of the six in $\varphi^{-1}(p)$, say $p_1$, $p_2$, and $p_3$, tend to $(L_1,L_2)$, and the other three points, \textit{i.e.}~$p_4$, $p_5$, and $p_6$, tend to $(L_2,L_1)$. Moreover, $\sigma$ maps the first triple to the second one; \textit{i.e.}~$\sigma(\{p_1,p_2,p_3\}) = \{p_4,p_5,p_6\}$. Thus $p_1$, $p_2$, and $p_3$ lie in the same connected component of $P$; in other words, the quotient of $P$ by the switching involution is path-connected. By Lemma~\ref{lem:singular-locus-P} we deduce that $P/\sigma$ is irreducible.

Hence, it is now sufficient to show that for some point $p_1\in P_{U_{\et}}$, the points $p_1$ and $\sigma(p_1)$ lie in the same connected component of $P_{U_{\et}}$. To do so, we consider one transposition $\tau$ in $\Mon_\psi$, whose existence is granted by Propositions~\ref{prop:transposition} and~\ref{cor:ram-simple}. If we consider it as an element of deck transformations, this means that there exist a point $R\in F\setminus \Branch(\wh \psi)$ and a loop $\gamma$ centered at $R$ such that $\gamma$ acts on the fibre of $R$ exchanging two points, say $L_1\mapsto L_2$ and $L_2\mapsto L_1$, and as the identity on the remaining  points of the fibre. 
We now consider the action under deck transformation of $\gamma$ for the natural map $P \dashrightarrow F$ defined by $\psi$. A moment's thought tells us that under the action of $\gamma$, the point $p=(L_1,L_2)\in P$, with $\psi(L_1)=\psi(L_2)=R$, is mapped to $(L_2,L_1)$.
Thus $P$ is irreducible, as we claimed.
\end{proof}

\renewcommand{\appendixname}{Appendix. Singular cubic surfaces}

\addappendix 

In this appendix we treat degenerations of Cayley cubics we introduced in Diagram \eqref{degeneration-surfaces}, focusing on lines, generalised twisted cubics, and the restrictions of the Voisin maps $\varphi$ and $\psi$ to a cubic surface. 

\subsection*{Singular cubic surfaces of type $\boldsymbol{4A_1}$}

Under a suitable change of coordinate, every cubic surface with four singular points of type $A_1$ can be expressed as zero locus of the polynomial
\begin{align*}
    t_0t_1t_2 + t_0t_1t_3 + t_0t_2t_3 + t_1t_2t_3.
\end{align*}
Its singular locus consists of the coordinate points
\begin{align*}
    p_0=(1:0:0:0),\quad p_1=(0:1:0:0),\quad p_2=(0:0:1:0),\quad p_3=(0:0:0:1).
\end{align*}
Such a surface can be obtained as the anticanonical model of the blow-up of $\mathbb P^2$ in a special configuration of points.
Indeed, let $\widetilde S$ be the blow-up of $\mathbb P^2$ in the six points given by the intersection of four general lines $L_0$, $L_1$, $L_2$, $L_3$.
Then the anticanonical linear system of $\widetilde S$ maps to a cubic surfaces $S$ and contracts the strict transforms of the $L_i$ to four singular $A_1$ points.
The exceptional divisors map to lines in a tetrahedron-like arrangement; we refer to them as edges and denote by $E_{ij}$ the line connecting the $\supth{i}$ and the $\supth{j}$ coordinate points.
In particular, the edge $E_{ij}$ is cut by the equations $(t_k=t_l=0)$, where $\{i,j,k,l\}=\{0,1,2,3\}$.
Besides the edges, there are three more lines on $S$ joining opposite edges; we denote by $J_{ij,kl}$ the line intersecting the edges $E_{ij}$ and $E_{kl}$, where $\{i,j,k,l\}=\{0,1,2,3\}$. These are images of the strict transforms of lines through the blown-up points corresponding to the edges they intersect. For example, the line $J_{01,23}$ is the zero locus $(t_0+t_1 = t_2+t_3=0)$.

Generalised twisted cubics on singular cubic surfaces have been studied in \cite[Section~2]{LLSvS}. Any generalised twisted cubic on an integral normal cubic surface moves in a 2-dimensional linear system and can be described as the image of the pullback of a line in $\mathbb P^2$ along the diagram of a determinantal representation of $S$. The determinantal representation for the Cayley cubic described above defines $K\in Z'(S)$, which contains, as reducible elements of the family, the union of all the edges passing through a fixed vertex of the tetrahedron:
 \[
 K = [E_{ij}\cup E_{ik} \cup E_{il}] \in Z'(S).
 \]

\begin{proof}[Proof of Lemma~\ref{lem: opposite edges}]
    The proof is elementary and can be checked directly by the definition of the Voisin map~$\varphi$. For example to compute $\varphi(E_{02}, E_{13})$, consider the plane $\pi=(t_1=0)$ containing the line $E_{02}$ and the point $p\in E_{13}$. Its intersection with $S$ is given by $\pi\cap S=E_{02}\cup E_{03}\cup E_{23}$; hence $\varphi(E_{02}, E_{13}) = \left[E_{03}\cup E_{13} \cup E_{23}\right]$ and the claim is shown.

    The other cases are completely similar.
\end{proof}

\subsection*{Singular cubic surfaces of type $\boldsymbol{2A_1A_3}$}
Let $S$ be a cubic surface with a singular point $P_2$ of type $A_3$ and two singular points $P_0$, $P_3$ of type $A_1$. It contains exactly five lines as depicted in Figure~\ref{fig:2A1A3}: the line $R_{23}$ through $P_2$, $P_3$; the line $R_{02}$ through $P_0$, $P_2$; the line $R_{03}$ through $P_{0}$, $P_3$; the line $L_1$ containing the only singular point $P_2$; and the line $L_2$ contained in the smooth locus.

Such a surface is given in an opportune coordinate system by the equation
    \[
    t_0t_2t_3 - t_1^2t_3-t_0t_1^2 = 0.
    \]
Its singular locus consists of the point $P_2:=(0:0:1:0)$ of type $A_3$ and of the points $P_0:=(1:0:0:0)$ and $P_3:=(0:0:0:1)$ of type $A_1$.\\
The lines are then described by the equations $R_{23}= (t_0=t_1=0)$, $R_{02}= (t_1=t_3=0)$, $R_{03}= (t_1=t_2=0)$, $L_1=(t_0=t_3=0)$, and $L_2=(t_2=t_0+t_3=0)$.

\begin{proof}[Proof of Proposition~\ref{prop:2A1A3-residual}]
        The proof can be checked through straightforward computations. For example, intersecting the surface $S$ with the plane $\pi=(t_0=0)$, one gets
        \begin{align*}
            \pi\cap  S &= (t_0=0)\cap \left(t_0t_2t_3 - t_1^2t_3-t_0t_1^2 = 0\right) \\
            &= (t_0=0)\cap\left(t_1^2t_3=0\right) = 2R_{23} + L_1,
        \end{align*}
        which shows that $L_1$ is residual to $R_{23}.$
\end{proof}
    
\subsection*{Singular cubic surfaces of type $\boldsymbol{A_1A_5}$}

A singular cubic surface $S$ with exactly one $A_1$ and one $A_5$ singularity is given in an opportune of coordinate system by the equation
    \[
    t_3\left(t_0t_2-t_1^2\right)+t_0^3 = 0, 
    \]
    whose singular locus consists of the two points $P_1:=(0:0:0:1)$ of type $A_5$ and $P_2:=(0:0:1:0)$ of type~$A_1$. There are exactly two lines on $S$; they have equations $L_1 = (t_0=t_1=0)$ and $L_2= (t_0=t_3=0)$.

\begin{proof}[Proof of Proposition~\ref{prop: lines-A1A5}]
    This follows from the straightforward computation
\begin{equation*}\pushQED{\qed}
    (t_3=0)\cap S = 3L_2, \quad
    (t_0=0)\cap S = L_1+2L_2.\qedhere \popQED
\end{equation*}
\renewcommand{\qed}{}     
\end{proof}

We now study non-normal cubic surfaces, whose analysis is necessary to prove Corollary~\ref{cor:lineSecTyp}.

\subsection*{The cubic surface $\boldsymbol{X_6}$}
In a suitable system of coordinates, the surface $X_6$ is given by the equation
\begin{align*}
    t_0^2t_2 + t_1^2t_3 = 0.
\end{align*}
It is a non-normal surface whose singular locus consists of the line $L=(t_0=t_1=0)$.

\begin{proposition}
    The line $L$ is not triple on $X_6$.
\end{proposition}

\begin{proof}
    The proof is an elementary computation:
    each plane containing $L$ is of the form $\pi_{(a:b)}=(at_0 + bt_1=0)$ with $(a:b)\in \mathbb P^1$.
    By considering the intersection with such a plane $\pi_{(-1:b)}$, one has 
    \begin{align*}
        X_6 \cap (-t_0 + bt_1 = 0)\ =\ \left(-t_0 + bt_1 = t_1^2\left(b^2t_2 + t_3\right) =0\right)\ =\ 2L + R_b
    \end{align*}
    with $R_b= \left(b^2t_2 + t_3 = bt_1 + t_0 = 0\right)$ different from $L$, and
    \begin{align*}
        X_6 \cap (t_1=0)\ =\ \left(t_1 = t_0^2t_2=0\right)\ =\ 2L + R 
    \end{align*}
    with $R= (t_1 = t_2 =0)$. This shows that $L$ is not triple.
\end{proof}

\subsection*{The cubic surface $\boldsymbol{X_7}$}
In a suitable system of coordinates, the surface $X_7$ is given by the equation
\begin{align*}
    t_0t_1t_2  + t_0^2t_3 + t_1^3 = 0.
\end{align*}
It is a non-normal surface whose singular locus consists of the line $L=(t_0=t_1=0)$.

\begin{proposition}
    There is a unique plane $\pi$ making $L$ triple, i.e.\ such that $\pi\cap X_7 = 3L$.
\end{proposition}

\begin{proof}
    The proof is an elementary computation. Each plane containing $L$ is of the form $\pi_{(a:b)} =(at_0+bt_1=0)$ with $(a:b)\in \mathbb P^1$.
    By considering the intersection with such a plane $\pi_{(a:-1)}$, one has
    \begin{align*}
        X_7 \cap (at_0 - t_1 = 0)\ =\ 
        \left( at_0 - t_1 = t_0^2\left(a^3t_0 + at_2 +t_3\right) = 
        0\right)\ =\
        2L + R_a
    \end{align*}
    with $R_a = (a^3t_0 + at_2 + t_3 = t_1 - at_0 = 0)$.
    Also, one has 
    \begin{align*}
        X_7 \cap (t_0 = 0)\ =\ \left(t_1^3 = t_0=0\right)\ = 3L.
    \end{align*}
    Hence $\pi=(t_0=0)$ is the only plane as in the claim.
\end{proof}

\subsection*{The cubic surface $\boldsymbol{X_8}$}

In a suitable system of coordinates, the surface $X_8$ is given by the equation
\begin{align*}
    t_1^3 + t_2^3 + t_1t_2t_3=0.
\end{align*}
It is a non-normal surface isomorphic to the cone over a nodal planar cubic whose singular locus consists of the line $L=(t_1=t_2=0)$.

\begin{proposition}
    There are two planes $\pi$ and $\pi'$ making $L$ triple, i.e.\ such that $\pi\cap X_8 = \pi'\cap X_8 = 3L$.
\end{proposition}

\begin{proof}
    The proof is an elementary computation: each plane containing $L$ is of the form $\pi_{(a:b)} =(at_1+bt_2=0)$ with $(a:b)\in \mathbb P^1$.
    By considering the intersection with such a plane $\pi_{(1:-b)}$, one has 
    \begin{align*}
        X_8 \cap (t_1 - bt_2 = 0)\ =\ \left(t_1 - bt_2 =  t_2^2\left(\left(b^3+1\right)t_2 + bt_3\right) = 0\right)\ =\ 2L + R_b
    \end{align*}
    with $R_b = (t_1 - bt_2 = (b^3+1)t_2 + bt_3=0)$.
    Also, one has 
    \begin{align*}
        X_8 \cap (t_2 = 0)\ =\ \left(t_1^3 = t_2=0\right)\ =\ 3L.
    \end{align*}
    The planes $\pi = (t_1=0)$ and $\pi'= (t_2=0)$ are hence the only planes as in the claim.
\end{proof}

%%%%%%%%%%%%%%%%%%%%%
% References
%%%%%%%%%%%%%%%%%%%%%

\newcommand{\etalchar}[1]{$^{#1}$}


\begin{thebibliography}{FMOS22+++} 

\bibitem[AG23]{AG}
N.~Addington and F.~Giovenzana, \emph{On the period of Lehn, Lehn, Sorger, and
  van Straten's symplectic eightfold}, Kyoto J.~Math.\
  \textbf{63} (2023), no.~1, 71--86, \doi{10.1215/21562261-2022-0033}.

\bibitem[AL17]{AL}
N.~Addington and M.~Lehn, \emph{On the symplectic eightfold associated to a
  Pfaffian cubic fourfold}, J.~reine angew.\ Math.\  \textbf{731} (2017), 129--137,  \doi{10.1515/crelle-2014-0145}.

\bibitem[Ame09]{Amerik}
E.~Amerik, \emph{A computation of invariants of a rational self-map}, Ann.\ Fac.\ Sci.\ Toulouse Math.~(6) \textbf{18} (2009),
  no.~3, 445-457, \doi{10.5802/afst.1211}.

\bibitem[BHP\etalchar{+}04]{BHPvV}
W.\,P.~Barth, K.~Hulek, C.\,A.\,M.~Peters, and A.~Van~de Ven, 
 \emph{Compact complex surfaces}, 2nd ed., Ergeb.\ Math.\ Grenzgeb.~(3), vol.~4, Springer-Verlag, Berlin,  2004, \doi{10.1007/978-3-642-57739-0}.
  
\bibitem[Bea80]{Beauville-nodal}
A.~Beauville,
\emph{Sur le nombre maximum de points doubles d'une surface dans {{\(\mathbb{P}^3\)}} ({{\(\mu(5)=31\)}})}, in: 
\emph{Journ{\'e}es de G{\'e}om{\'e}trie Alg{\'e}brique d'Angers, Juillet 1979}, pp.~207--215, Sijthoff \& Noordhoff, Alphen aan den Rijn---Germantown, Md., 1980.
  
\bibitem[BD85]{BD}
A.~Beauville and R.~Donagi, 
\emph{La vari{\'e}t{\'e} des droites d'une hypersurface cubique de dimension $4$}, C.\ R.\ Acad.\ Sci.\ Paris S{\'e}r.~I Math.\ \textbf{301} (1985), no.~14, 703--706.

\bibitem[BFK{\etalchar{+}}24]{7auth}
M.~Bernardara, E.~Fatighenti, G.~Kapustka, M.~Kapustka, L.~Manivel, G.~Mongardi, and F.~Tanturri, 
\emph{Even nodal surfaces of {K}3 type}, preprint 
 \arXiv{2402.08528} (2024).
 
\bibitem[Bru80]{bruce}
J.\,W.~Bruce,
\emph{A stratification of the space of cubic surfaces}, 
Math.\ Proc.\ Cambridge Philos.\ Soc.\ \textbf{87}
  (1980), no.~3, 427--441, \doi{10.1017/s0305004100056863}.

\bibitem[Cat81]{Catanese}
F.~Catanese, 
\emph{Babbage's conjecture, contact of surfaces, symmetric determinantal varieties and applications}, Invent.\ Math.\  \textbf{63} (1981), no.~3, 433--465, \doi{10.1007/bf01389064}.

\bibitem[CCC{\etalchar{+}}24]{catanese-new}
F.~Catanese, in~collaboration~with Y.~Cho, S.~Coughlan, D.~Frapporti, A.~Verra, M.~Kiermaier, and S.~Kurz, 
\emph{Varieties of nodal surfaces, coding theory and discriminants of cubic hypersurfaces. Part~1: Generalities and nodal {K}3 surfaces. Part 2: Cubic hypersurfaces, associated discriminants. Part 3: Nodal quintics. Part 4: Nodal sextics}, preprint \arXiv{2206.05492v3} (2024).

\bibitem[Che21]{chen}
H.~Chen, 
\emph{The {V}oisin map via families of extensions}, 
Math.~Z.\ \textbf{299} (2021), no.~3-4, 1987--2003,
  \doi{10.1007/s00209-021-02747-1}.

\bibitem[CG72]{griffiths-clemens}
C.\,H.~Clemens and P.\,A.~Griffiths, 
\emph{The intermediate {J}acobian of the cubic threefold},
Ann.\ of Math.~(2) \textbf{95} (1972), 281--356, 
 \doi{10.2307/1970801}.

\bibitem[Dol12]{CAG}
I.~Dolgachev, 
 \emph{Classical algebraic geometry. A modern view}, 
 Cambridge Univ.\ Press, Cambridge, 2012,
 \doi{10.1017/CBO9781139084437}.

\bibitem[FMOS22]{antisymplectic}
L.~Flapan, E.~Macr\`\i, K.\,G.~O'Grady, and G.~Sacc\`a, 
\emph{The geometry of antisymplectic involutions, {I}}, 
Math.~Z.\ \textbf{300} (2022), no.~4, 3457--3495,
\doi{10.1007/s00209-021-02909-1}.

\bibitem[FMOS26]{FMSOG-II}
\bysame, 
\emph{The geometry of antisymplectic involutions, {II}},
J. \'Ec.\ polytech.\ Math.\ \textbf{13} (2026), 809--851,
\doi{10.5802/jep.337}.

\bibitem[Gio26]{mio}
F.~Giovenzana, 
\emph{On the uniruled {V}oisin divisor on the {LLS}v{S} variety},
Beitr.\ Algebra Geom., published online on May 6, 2026, to appear in
print, \doi{10.1007/s13366-026-00851-z}.

\bibitem[GK23]{GK-invariants}
F.~Gounelas and A.~Kouvidakis, 
\emph{On some invariants of cubic fourfolds}, 
Eur.~J.\ Math.\ \textbf{9} (2023), no.~3, Paper No.~58, 
 \doi{10.1007/s40879-023-00651-y}.

\bibitem[GK24]{GK-lines}
  \bysame, 
  F.~Gounelas and A.~Kouvidakis, \emph{Geometry of {L}ines on a {C}ubic {F}our-{F}old}, Int.\ Math.\ Res.\ Not.\ IMRN \textbf{2024}, no.~2,
  1373--1421, \doi{10.1093/imrn/rnac160}.

\bibitem[GK25]{GK-monodromy}
  \bysame, 
\emph{The Fermat cubic and monodromy of lines}, 
New York J.~Math.\ \textbf{31} (2025), 650--667.

\bibitem[GKP16]{GKP}
D.~Greb, S.~Kebekus, and T.~Peternell, 
\emph{{\'E}tale fundamental groups of {Kawamata} log terminal spaces, flat sheaves, and quotients of abelian varieties}, 
Duke Math.~J.\  \textbf{165} (2016), no.~10, 1965--2004, 
  \doi{10.1215/00127094-3450859}.
  
\bibitem[Har79]{Har}
J.~Harris, 
 \emph{Galois groups of enumerative problems},
 Duke Math.~J.\ \textbf{46} (1979), no.~4, 685--724, 
 \doi{10.1215/S0012-7094-79-04635-0}.

\bibitem[Huy23]{HuyBookCubics}
D.~Huybrechts, 
 \emph{The Geometry of Cubic Hypersurfaces},
 Cambridge Stud.\ Adv.\ Math., vol.~206, Cambridge Univ.\ Press, Cambridge, 2023, \doi{10.1017/9781009280020}.

\bibitem[Huy24]{HUY-nodal-quintics}
\bysame,
\emph{Nodal quintic surfaces and lines on cubic fourfolds} (with an appendix by J.\,C.~Ottem), 
 Enseign.\ Math., \textbf{70} (2024), no.~3-4, 499--532,  
 \doi{10.4171/lem/1063}.
 
\bibitem[Leh18]{Lehn}
C.~Lehn,
\emph{Twisted cubics on singular cubic fourfolds---on {Starr}'s fibration}, 
Math.~Z.\ \textbf{290} (2018), no.~1-2, 379--388, 
 \doi{10.1007/s00209-017-2021-x}.

\bibitem[LLS\etalchar{+}17]{LLSvS}
C.~Lehn, M.~Lehn, C.~Sorger, and D.~van Straten, 
\emph{Twisted cubics on cubic fourfolds}, 
 J.~reine angew.\ math.\ \textbf{731} (2017), 87--128,
 \doi{10.1515/crelle-2014-0144}.
 
\bibitem[Leh15]{Lehn-oberwolfach}
M.~Lehn, 
\emph{Twisted cubics on a cubic fourfold and in involution on the associated 8-dimensional symplectic manifold}, in: \emph{Mini-workshop: Singular curves on $K3$ surfaces and hyperk\"ahler manifolds}, pp.~2962--2964, 
Oberwolfach Rep.\ \textbf{12} (2015), no.~4. 
\doi{doi.org/10.4171/OWR/2015/51}. 

\bibitem[LPZ23]{LPZ}
C.~Li, L.~Pertusi, and X.~Zhao,
 \emph{Twisted cubics on cubic fourfolds and stability conditions}, 
 Algebr.\ Geom.\ \textbf{10} (2023), no.~5, 620--642, 
 \doi{10.14231/AG-2023-022}.

\bibitem[Liu02]{Liu}
Q.~Liu, 
 \emph{Algebraic geometry and arithmetic curves} (translated from the French by R.~Ern\'{e}), Oxf.\ Grad.\ Texts Math., vol~6,  Oxford Univ.\ Press, Oxford, 2002.
 
\bibitem[Mur20]{Muratore}
G.\,E.~Muratore, 
 \emph{The indeterminacy locus of the {V}oisin map}, 
Beitr.\ Algebra Geom.\ \textbf{61} (2020), no.~1, 73--88, 
 \doi{10.1007/s13366-019-00454-x}.

\bibitem[NO21]{Nesterov-Oberdieck}
D.~Nesterov and G.~Oberdieck,
\emph{Elliptic curves in hyper-{K}\"{a}hler varieties}, 
Int.\ Math.\ Res.\ Not.\ IMRN \textbf{2021}, no.~4, 2962--2990,
 \doi{10.1093/imrn/rnaa016}.
 
\bibitem[SY25]{sottile-yahl}
F.~Sottile and T.~Yahl,
\emph{Galois groups in enumerative geometry and applications},
Expo.\ Math.\ \textbf{43} (2025), no.~6, Paper No.~125731,
\doi{10.1016/j.exmath.2025.125731}. 
  
\bibitem[Vak06]{vakil}
R.~Vakil, 
\emph{Schubert induction}, 
Ann.\ of Math.~(2) \textbf{164} (2006), no.~2, 489--512, 
 \doi{10.4007/annals.2006.164.489}.

\bibitem[Voi86]{voisin-torelli}
C.~Voisin, 
\emph{Th\'eor\`eme de {T}orelli pour les cubiques de {$\mathbb P^5$}}, 
Invent.\ Math.\ \textbf{86} (1986), no.~3, 577--602,
 \doi{10.1007/BF01389270}.

\bibitem[Voi04]{voisin-map-F}
\bysame, 
\emph{Intrinsic pseudo-volume forms and {$K$}-correspondences}, 
 in: \emph{The {F}ano {C}onference}, pp.~761--792. Univ. Torino, Turin, 2004.

\bibitem[Voi16]{Voisin-map-varphi}
\bysame,
\emph{Remarks And Questions On Coisotropic Subvarieties and 0-Cycles of Hyper-K{\"a}hler Varieties}, in: \emph{K3 Surfaces and their Moduli}, pp.~365--399, Progr.\ Math., vol.~315, Birkh\"auser/Springer, Cham, 2016, 
 \doi{10.1007/978-3-319-29959-4_14}.

\end{thebibliography}
\end{document}